\documentclass[11pt,a4paper]{amsart}[2000]
\title[Nuclear Dimension of Partial Crossed Products]{Nuclear Dimension of Crossed Products associated to Partial Dynamical Systems}
\author{Shirly Geffen}
\address{Department of Mathematics, Ben Gurion University of the Negev, Be'er Sheva, Israel. Mathematisches Institut, WWU M\"{u}nster, Germany.}
\email{shirlyg@post.bgu.ac.il}
\thanks{The author was supported by Israel Science Foundation grant no. 476/16, a 		Kreitman foundation fellowship, a Minerva
	fellowship programme, partially supported by Deutsche Forschungsgemeinschaft (DFG, German Research Foundation) through SFB 878 and under Germany’s Excellence Strategy – EXC 2044 – 390685587, Mathematics Münster – Dynamics – Geometry – Structure, and also partially supported by ERC Advanced Grant 834267 - AMAREC}

\usepackage{amsmath}
\usepackage{hyperref, aliascnt}
\usepackage{amssymb}
\usepackage{amsthm}
\usepackage[all]{xypic}
\usepackage{enumitem}
\usepackage{mathtools}
\usepackage{color}
\usepackage{tikz}
    	
\def\firstcircle{(70:0.9cm) circle (1.2cm)}
\def\secondcircle{(190:0.9cm) circle (1.2cm)}
\def\thirdcircle{(308:0.9cm) circle (1.2cm)}

\theoremstyle{plain}

\theoremstyle{definition}
\newtheorem{lma}{Lemma}[section]

\newaliascnt{ThmCt}{lma}
\newtheorem{Thm}[ThmCt]{Theorem}
\aliascntresetthe{ThmCt}

\newaliascnt{ClaimCt}{lma}
\newtheorem{Claim}[ClaimCt]{Claim}
\aliascntresetthe{ClaimCt}

\newaliascnt{CorCt}{lma}
\newtheorem{Cor}[CorCt]{Corollary}
\aliascntresetthe{CorCt}

\newaliascnt{ObsCt}{lma}
\newtheorem{Obs}[ObsCt]{Observation}
\aliascntresetthe{ObsCt}

\newaliascnt{ProbCt}{lma}
\newtheorem{Prob}[ProbCt]{Problem}
\aliascntresetthe{ProbCt}

\newaliascnt{PropCt}{lma}
\newtheorem{Prop}[PropCt]{Proposition}
\aliascntresetthe{PropCt}

\newtheorem*{Thm*}{Theorem}
\newtheorem*{Cor*}{Corollary}
\newtheorem*{Prop*}{Proposition}

\newaliascnt{DefCt}{lma}
\newtheorem{Def}[DefCt]{Definition}
\aliascntresetthe{DefCt}

\newaliascnt{RmkCt}{lma}
\newtheorem{Rmk}[RmkCt]{Remark}
\aliascntresetthe{RmkCt}

\newaliascnt{ExlCt}{lma}
\newtheorem{Exl}[ExlCt]{Example}
\aliascntresetthe{ExlCt}

\numberwithin{equation}{section}

\newcommand{\id}{\mathrm{id}}
\newcommand{\dimnuc}{\mathrm{dim}_{\mathrm{nuc}}}

\newcommand{\dr}{\mathrm{dr}}
\newcommand{\Prim}{\mathrm{Prim}}
\newcommand{\Aut}{\mathrm{Aut}}
\newcommand{\LCH}{locally compact Hausdorff }

\begin{document}
\begin{abstract}
	We bound the nuclear dimension of crossed products associated to some partial actions of finite groups or $\mathbb{Z}$ on finite dimensional locally compact Hausdorff second countable spaces. Our results apply to globalizable partial actions, finite group partial actions, minimal partial automorphisms, and partial automorphisms acting on zero-dimensional spaces, or a class of one dimensional spaces, containing $1$-dimensional CW complexes. This extends work on global systems by Hirshberg and Wu.
\end{abstract}
\maketitle

\section{Introduction}\label{s:Int}

In 1989 Elliott conjectured that separable simple nuclear $C^*$-algebras can be classified by an invariant consisting of $K$-theoretic data. Elliott's program (see \cite{RorYBook} for an overview) was a major focus of research involving dozens of researchers over a period of decades. One major milestone was the Kirchberg-Phillips classification in the setting of purely infinite $C^*$-algebras satisfying the UCT \cite{Phillips-PurelyInf, KirPhi}.
An increasingly large classes of $C^*$-algebras were classified also in the stably finite case (see \cite{EllToms} for a survey). However, a counterexample was constructed by Toms in \cite{Toms1}. This led to searching for a dividing line between ``nice'' $C^*$-algebras which can be classified and ones which are not. The key concept turned out to be nuclear dimension, which in the past few years was shown to be the missing ingredient. This concept is a noncommutative analogue of topological covering dimension, in the sense that $\dimnuc(C_0(X))=\dim(X)$ for $X$ a second-countable \LCH space. With this notion in hand, the classification of unital simple separable $C^*$-algebras with finite nuclear dimension which satisfy the UCT has been completed \cite{EGLN, GLN1, GLN2}.

This is one reason why extending the class of $C^*$-algebras which are known to be of finite nuclear dimension is an important task. Though lower bounds for this invariant are usually not easy to find, in most cases we only need to know that upper bounds exist.

The notion of nuclear dimension was first introduced by Winter and Zacharias \cite{Winter-Zacharias}. In that paper they pose the following problem.

\begin{Prob}\cite[Problem~9.4.]{Winter-Zacharias}\label{pr:FinDimNucCrossProd}
	Let $G$ be a group, $A$ be a $C^*$-algebra and $\alpha: G\to \Aut(A)$ be a group homomorphism. Find conditions on $A$, $G$ and $\alpha$, under which $\dimnuc(A\rtimes_{\alpha}G)$ is finite.
\end{Prob}

One strategy is to place the main focus on conditions on $\alpha$, see \cite{SWZ, HSWW, Gar, GWY, HWZ}.
For example, Hirshberg, Winter and Zacharias develop in \cite{HWZ} the theory of Rokhlin dimension for finite groups and integers actions on unital separable $C^*$-algebras (extended in \cite{HP} to the non-unital case). As a result, they show that the nuclear dimension of these crossed products is bounded by a polynomial in the Rokhlin dimension of the action and the nuclear dimension of the base algebra.
In \cite{Gar} Gardella extends the concept of Rokhlin dimension to compact group actions and proves the corresponding result for the crossed products \cite[Theorem~3.4]{Gar2}. He shows that in the commutative case, an action by a compact Lie group has finite Rokhlin dimension if and only if it induces a free action on the spectrum \cite[Theorem~4.2]{Gar}, and raises the question whether this is the case also for arbitrary compact groups \cite[Question~5.3]{Gar}. The following example, that can be deduced from \cite{Lev}, answers this question negatively. Moreover, to the author's knowledge, this is a first example of a compact group acting freely on a finite dimensional compact space such that the associated crossed product has infinite nuclear dimension.

\begin{Exl}\label{e:Levin}
	Let $G$ and $X$ be as in \cite[Theorem~1.1]{Lev}. Then $G$ is a Cantor group acting freely on a $1$-dimensional compact metric space $X$. Green's imprimitivity Theorem \cite[Corollary~4.11]{Will} implies strong Morita equivalence between $C(X)\rtimes G$ and $C(X/G)$. Therefore, $\dimnuc(C(X)\rtimes G)=\dim(X/G)=\infty$.
\end{Exl}

Another example, of a circle action on a (non commutative) $C^*$-algebra with finite nuclear dimension such that the crossed product has infinite nuclear dimension can be deduced as follows. Take $\alpha$ to be a minimal homeomorphism on a compact infinite dimensional metric space $X$ such that $\alpha$ has mean dimension zero (existence follows from \cite[Section~3]{DPS}). The results of \cite{ElliottNiu} imply $\dimnuc(C(X)\rtimes_\alpha\mathbb{Z})<\infty$. Letting $\hat{\alpha}$ denote the dual action of the circle, we have
Morita equivalence between $C(X)$ and the double crossed product. Thus,
$\dimnuc(C(X)\rtimes_\alpha\mathbb{Z}\rtimes_{\hat{\alpha}}\mathbb{T})=\infty$.

However, examples with $G$ finite group, or the group of integers acting on a $C^*$-algebra $A$ with finite nuclear dimension such that $\dimnuc(A\rtimes G)=\infty$ are still not known.

Considering $\mathbb{Z}$ actions on commutative $C^*$-algebras, it is shown in \cite{HWZ} that automorphisms arising from minimal homeomorphisms of finite dimensional compact metrizable spaces always have finite Rokhlin dimension. The latter result has been generalized shortly afterwards by Szab\'{o} \cite{Sza}. In contrast to the compact group case, he proves that $\mathbb{Z}$-actions on compact metrizable finite dimensional spaces are free if and only if they have finite Rokhlin dimension. In particular, this answers \autoref{pr:FinDimNucCrossProd} for such systems.
Hence, known results concerning Rokhlin dimension cannot be used to bound nuclear dimension in the non-free case.

The next remarkable advance was made by Hirshberg and Wu, who proved that for any \LCH dynamical system $(X,h)$, the associated crossed product $C^*$-algebra has nuclear dimension bounded by a polynomial in $\dimnuc(C_0(X))$ \cite{HW}. This answers \autoref{pr:FinDimNucCrossProd} in case $G=\mathbb{Z}$ and $A$ is a commutative $C^*$-algebra.
This result of Hirshberg and Wu is the main motivation for this paper. 

We consider crossed products associated to partial dynamical systems, and primarily to partial automorphisms. A partial automorphism of a topological space $X$ is a triple $\Theta=(\theta,U,V)$
consisting of two open subsets $U$ and $V$ in $X$, and a homeomorphism $\theta: U\to V$. 
Given such triple, one can associate a crossed product $C^*$-algebra, which coincides with the original construction when $U=V=X$.
The theory of partial actions was developed by Exel (see \cite{Exel}), and gives new examples of crossed products that cannot be constructed using global systems. For example, Exel shows in \cite{Ex} that every AF algebra is a crossed product by a partial automorphism acting on a zero dimensional space. However, $K$-theory obstructions imply that unital AF algebras are never crossed products by global automorphisms. Also the Jiang-Su algebra, which has a central role in the classification program (see \cite{RorWin, CETWW}), cannot be viewed as a crossed product by a global automorphism, due to similar $K$-theory obstructions. It is shown in \cite[Remarks~2.6]{DPS18} that the Jiang-Su algebra is a crossed product by a partial automorphism on a zero-dimensional space.

The goal of this paper is to find conditions on topological partial actions, under which the crossed product has finite nuclear dimension. We organize the content as follows. 

We collect some preliminaries in \autoref{s:Pre}.

In \autoref{s:Glob} we deal with partial systems that globalize to honest actions and use the properties that are preserved between the two systems in order to bound the nuclear dimension of the associated crossed products. As an application, we consider partial automorphisms that are a direct limit of globalizable ones. We show that partial automorphisms on zero-dimensional spaces are of this form.

\autoref{s:FinGrp} deals with partial actions of finite groups.

\autoref{s:SupDom} deals with partial automorphisms that have uniformly bounded orbits. A global action has this property if and only if it is cyclic. 

\autoref{s:DenDom} contains a technical result which reduces the problem to the case of partial automorphisms with dense domains.

In \autoref{s:MinParAut} we bound the nuclear dimension (and more precisely, the decomposition rank) of crossed products associated to minimal partial homeomorphisms that can be extended to the whole space. We then use methods developed in \autoref{s:DenDom} to extend this result to general minimal partial automorphisms. We conclude that $C(X)\rtimes_\theta\mathbb{Z}$ is classifiable whenever $\theta$ is a minimal partial automorphism acting on a finite dimensional infinite compact Hausdorff second countable space $X$.

In \autoref{s:ZeroDimBdrsSubR} we again use ideas from \autoref{s:DenDom} in order to deal with partial automorphisms acting on $1$-dimensional spaces with the property that every closed subset has a zero-dimensional boundary (e.g. $1$-dimensional manifolds).

We believe that the results in this paper should hold for partial automorphisms acting on finite dimensional \LCH spaces, with no further restrictions. We do not know whether the technique of decomposing a partial automorphism into tractable pieces and then resort to known results about global systems will prove useful in higher dimensions. Alternatively, one could try to directly mimic the approach of Hirshberg and Wu, after developing an appropriate Rokhlin theory for partial automorphisms (see \cite{AGG2} for the concept of Rokhlin dimension for finite group partial actions).

\begin{center}
	\Large{A}\normalsize{CKNOWLEDGEMENTS}
\end{center}
The author is grateful to her supervisors Prof. Ilan Hirshberg, who suggested to her this interesting problem, and Prof. Wilhelm Winter, for their generous support and helpful advice. 
The author is also grateful to the referee for carefully reading the manuscript, and for giving many helpful comments.

\section{Preliminaries}\label{s:Pre}
We start by recalling definitions and basic properties of topological and $C^*$-algebraic partial actions of discrete groups. For more details, we refer to \cite{Exel}.

\begin{Def}\label{d:AlgPar}
	A \textit{$C^*$-algebraic partial action} $\theta$ of a discrete group $G$ on a $C^*$-algebra $A$ is a pair $\theta=(\{D_g\}_{g\in G}, \{\theta_g\}_{g\in G})$, consisting of ideals $\{D_g\}_{g\in G}$ in $A$ and $^*$-isomorphisms $\{\theta_g: D_{g^{-1}}\to D_g\}_{g\in G}$ such that:
	\begin{enumerate}
		\item $D_{1}=A$ and $\theta_{1}=\id_A$;
		\item $\theta_g\circ \theta_h\subseteq \theta_{gh}$ for all $g,h\in G$.
	\end{enumerate}
\end{Def}

\begin{Def}\label{d:TopPar}
	A \textit{topological partial action} $\theta$ of a discrete group $G$ on a topological space $X$ is a pair $\theta=(\{D_g\}_{g\in G}, \{\theta_g\}_{g\in G})$, consisting of open sets $\{D_g\}_{g\in G}$ in $X$ and homeomorphisms $\{\theta_g: D_{g^{-1}}\to D_g\}_{g\in G}$ such that:
	\begin{enumerate}
		\item $D_{1}=X$ and $\theta_{1}=\id_X$;
		\item $\theta_g\circ \theta_h\subseteq \theta_{gh}$ for all $g,h\in G$.
	\end{enumerate}
\end{Def}

\begin{Obs}\label{o:DomRel}
	For a partial action $\theta=(\{D_g\}_{g\in G}, \{\theta_g\}_{g\in G})$, the equality
	
	$$\theta_g(D_h\cap D_{g^{-1}})=D_{gh}\cap D_g,$$
	holds for all $g,h\in G$.
\end{Obs}

The contravariant functor between \LCH spaces, with proper continuous maps and commutative $C^*$-algebras with non-degenerate $^*$-homomorphi-\\sms:
$$X\rightsquigarrow C_0(X)$$
gives a natural bijection between topological partial actions on \LCH spaces, and $C^*$-algebraic partial actions on commutative $C^*$-algebras. Therefore, we will sometimes abuse notation and write $\theta$ both for the partial action on $X$ and the induced partial action on $C_0(X)$.\\

\textit{Partial automorphisms }(also called \textit{semi-saturated partial actions of $\mathbb{Z}$}) form an important family of partial actions of $\mathbb{Z}$. We recall the definition:

\begin{Def}\label{d:ParAut}
	Let $X$ be a \LCH space and let $U,V$ be open subsets of $X$. Assume that $\theta: U\to V$ is a given homeomorphism. Then $\theta$ induces a partial action of $\mathbb{Z}$ on $X$, defined by $D_n\coloneqq\mathrm{Dom}(\theta^{-n})$ and $\theta_n\coloneqq\theta^n$, for every $n\in \mathbb{Z}$.
\end{Def}

\begin{Rmk}\label{r:IncDom}
	Let $\theta$ be a partial automorphism. It is easy to see that
	$$\ldots\subseteq D_2\subseteq D_1\subseteq D_0=X,$$
	and
	$$\ldots\subseteq D_{-2}\subseteq D_{-1}\subseteq D_0=X.$$
\end{Rmk}

\begin{Def}\label{d:InvSetMinParAct}
	Let $\theta=(\{D_g\}_{g\in G}, \{\theta_g\}_{g\in G})$ be a partial action of $G$ on a topological space $X$. We say that $Y\subseteq X$ is a \textit{$\theta$-invariant subset} if
	\[\theta_g(Y\cap D_{g^{-1}})\subseteq Y, \text{ for all } g\in G. \]
	
	A $\theta$-invariant subset $Y$ gives rise to a partial actions of $G$ on $Y$, by intersecting all domains with $Y$, and restricting the maps to the new domains.\\
	Note that if $\theta$ is a $\mathbb{Z}$-partial action generated by a partial automorphism, then $Y\subseteq X$ is $\theta$-invariant if and only if $\theta(Y\cap D_{-1})\subseteq Y$ and $\theta^{-1}(Y\cap D_1)\subseteq Y$.\\
	We say that $\theta$ is \textit{minimal} if $X$ and $\emptyset$ are the only open (equivalently, closed) $\theta$-invariant subsets in $X$.
\end{Def}

We recall the construction of a\textit{ partial crossed product}:\\
Let $\theta=(\{\theta_g\}_{g\in G}, \{D_g\}_{g\in G})$ be a partial action of a discrete group $G$ on a $C^*$-algebra $A$. Define the \textit{algebraic partial crossed product}, $A\rtimes_{\theta,alg}G$, to be finitely supported functions from $G$ to $A$, i.e. $\sum_{g\in G} a_g \delta_g$, with finitely many non-zero coefficients $a_g$, and such that $a_g\in D_g$ for all $g\in G$.\\
Addition and scalar multiplication are defined as usual.\\
Multiplication and involution are determined by:\\
$$(a\delta_g)(b\delta_h) \coloneqq \theta_g(\theta_{g^{-1}}(a)b)\delta_{gh} \ \text{ and } (a\delta_g)^*\coloneqq\theta_{g^{-1}}(a^*)\delta_{g^{-1}},$$
for all $a\in D_g$ and $b\in D_h$. One checks that it is well-defined.
Extend linearly to obtain an associative $^*$-algebra, $A\rtimes_{\theta,alg}G$,  see \cite{DokuchaevExel}.
Finally, the \textit{partial crossed product} $A\rtimes_{\theta}G$ is the $C^*$-algebra obtained by completing $A\rtimes_{\theta,alg}G$ with respect to the semi-norm (which turns out to be a norm): 
$$||x||_{\max}\coloneqq \sup\{ ||\pi(x)|| : \pi \ \mathrm{is\  a}\  ^*\mathrm{-representation \ of \ } A\rtimes_{\theta,alg}G \}.$$

We continue by recalling the main definitions and results concerning to enveloping actions. For further discussion, we refer to \cite{Abadie}.

\begin{Def} \label{d:TopGlob}
	A topological partial action $\theta$ of $G$ on $X$ is called \textit{globalizable}, if there exist a topological space $Y$ and an action $\alpha$ of $G$ on $Y$ such that:
	\begin{enumerate}
		\item $X\subseteq Y$ is an open subset;
		\item $Y=\bigcup_{g\in G}\alpha_g(X)$;
		\item $D_g=\alpha_g(X)\cap X$;
		\item $\theta_g(x)=\alpha_g(x)$ for all $x\in D_{g^{-1}}$.
	\end{enumerate} 
\end{Def}	

\begin{Rmk}\label{r:PropTopGolob}
	\mbox{}
	\begin{enumerate}
		
		\item Topological partial actions are always uniquely (up to a strong notion of equivalence) globalizable.
		\item If $X$ is locally compact (resp. compact, resp. second countable) then $Y$ is locally compact (resp. compact, resp. second countable). This follows from the construction of $Y$ as an open quotient of the product space $X\times G$.
		\item $Y$ may not be Hausdorff, even if $X$ is. A Hausdorff globalization exists if and only if the graph, $\mathrm{Graph}(\theta)\coloneqq \{(\theta_g(x),g,x): x\in D_{g^{-1}}\}$, is a closed subset of $X\times G\times X$.
		\item Let $X$ be a Hausdorff space. If the domains $D_g$ are clopen in $X$, then the graph is closed. 
		The converse direction holds whenever $X$ is compact.
	\end{enumerate}
\end{Rmk}

\begin{Def}\label{d:AlgGlob}
	A $C^*$-algebraic partial action $\theta$ of $G$ on $A$ is called \textit{globalizable}, if there exist a $C^*$-algebra $B$ and an action $\alpha$ of $G$ on $B$, such that:
	\begin{enumerate}
		\item $A$ is an ideal in $B$;
		\item $B=\overline{\sum_{g\in G} \alpha_g(A)}$;
		\item $D_g=\alpha_g(A)\cap A$;
		\item $\theta_g(x)=\alpha_g(x)$ for all $x\in D_{g^{-1}}$.
	\end{enumerate}
\end{Def}

\begin{Rmk}\label{r:PropAlgGlob}
	\begin{enumerate}
		\mbox{}
		\item Globalizations do not always exist in the $C^*$-context.
		\item If $A$ is unital then a globalization exists if and only if the domains $D_g$ are all unital.
		\item For a \LCH space $X$, and $C^*$-algebraic partial action on $C_0(X)$, a globalization exists if and only if a Hausdorff globalization exists for the corresponding topological partial action on $X$.
	\end{enumerate}
\end{Rmk}

\begin{Thm}(\cite[Theorem~4.18]{Abadie}) \label{t:MorEqGlob}
	Let $\theta$ be a globalizable partial action of a discrete group $G$ on a $C^*$-algebra $A$. Denote the globalization by $\alpha$, acting on $B$.\\ Then $A\rtimes_{\theta}G$ is a full-hereditary $C^*$-subalgebra inside $B\rtimes_{\alpha}G$.
\end{Thm}

As in the case of honest actions, invariant open sets of a topological system give rise to closed two-sided ideals in the crossed product algebra.

\begin{Prop}\cite[Theorem~29.9]{Exel} \label{p:InvSetS.e.s}
	Let $\theta$ be a partial action of a discrete group $G$ on a  \LCH space $X$.
	Let $U$ be an open $\theta$-invariant subset of $X$. Then $$0\to C_0(U)\rtimes_{\theta|_{U}} G\to C_0(X)\rtimes_{\theta}G\to C_0(X\setminus U)\rtimes_{\theta|_{X\setminus U}}G\to 0$$ is a short exact sequence of $C^*$-algebras.
\end{Prop}

We review some well-known facts related to dimension theory.

\begin{Thm} \cite[Chapter~3, Proposition~6.4]{Pears} \label{t:DimSub}
	Let $Y$ be any subspace of a \LCH second countable space $X$. Then $\dim(Y)\leq \dim(X)$.
\end{Thm}
%since LCH 2nd ctble is metrizable so totally normal
\begin{Thm} \cite[Chapter~3, Proposition~5.3]{Pears} \label{t:DimF-Sig}
	Let $X$ be a \LCH second countable space such that $X=\bigcup_{n\in\mathbb{N}} A_n$, where each $A_n$ is an $\mathrm{F}_\sigma$-subset. Then 
	$$\dim(X)=\max_n\{\dim(A_n)\}.$$
\end{Thm}
% again, every LCH 2nd ctble is totally normal and normal, so apply Winter's remarks

We recall the definition for nuclear dimension of $C^*$-algebras and permanence properties that will be useful for us. For further discussion, see \cite{Winter-Zacharias}.

\begin{Def}\label{d:DimNucDr}
	A $C^*$-algebra $A$ has \textit{nuclear dimension} at most $n$, $\dimnuc(A)\leq n$, if there exists a net $(F_{\lambda},\psi_{\lambda}, \varphi_{\lambda})_{\lambda\in \Lambda}$ consisting of: $F_{\lambda}$ finite dimensional $C^*$-algebras, and $\psi_{\lambda}:A\to F_{\lambda}$, $\varphi_{\lambda}: F_{\lambda}\to A$ completely positive maps, satisfying:  
	\begin{enumerate}
		\item $\varphi_{\lambda}\circ\psi_{\lambda}(a)\to a$ uniformly on finite subsets of $A$;
		\item $\psi_{\lambda}$ are contractive;
		\item $F_{\lambda}$ decomposes as $F_{\lambda}=F_{\lambda}^{(0)}\oplus\ldots\oplus F_{\lambda}^{(n)}$ such that the restriction of $\varphi_{\lambda}$ to each $F_{\lambda}^{(i)}$, $i\in \{0,1,\ldots,n\}$, is a contractive, order zero map.
	\end{enumerate}
	
	If one can moreover arrange that the maps $\varphi_{\lambda}$ are contractive, then we say that the \textit{decomposition rank} of $A$ is at most $n$, $\dr(A)\leq n$.
\end{Def}

We recall some permanence properties for decomposition rank and nuclear dimension. We refer the reader to \cite{KW, Winter-Zacharias}.

\begin{Prop}\label{p:PerPropDimNucDr}
	\mbox{}
	\begin{enumerate}
		\item Suppose $A=\varinjlim{A_n}$. Then \[\dimnuc(A)\leq \mathrm{lim inf}_n(\dimnuc(A_n)).\]
		\item Let $J$ be an ideal in $A$, then $$\dimnuc(A/J) \leq 
		\dimnuc(A)\leq \dimnuc(J)+\dimnuc(A/J)+1.$$
		Moreover, $\dr(A/J)\leq \dr(A)$.
		\item If $B$ is a hereditary $C^*$-subalgebra in $A$, then $$\dimnuc(B)\leq \dimnuc(A) \text{ and } \dr(B)\leq \dr(A).$$ 
		If $B$ is full-hereditary in $A$, then $$\dimnuc(B)=\dimnuc(A) \text{ and } \dr(B)=\dr(A).$$
		\item Let $X$ be a \LCH second countable space. Then
		$$\dimnuc(C_0(X))=\dr(C_0(X))=\dim(X).$$
	\end{enumerate}
	
\end{Prop}

Finally, our motivation is to push further the following theorem, by Hirshberg and Wu, to the context of partial actions.

\begin{Thm}\cite[Theorem 5.1]{HW}  \label{t:HirWu}
	Let $X$ be a \LCH space and $\hat{\alpha}\in \mathrm{Homeo}(X)$. Denote by $\alpha: \mathbb{Z}\to \Aut(C_0(X))$ the generated action. Then
	$$\dimnuc(C_0(X)\rtimes_{\alpha}\mathbb{Z})\leq 2\dim(X)^2+6\dim(X)+4.$$
\end{Thm}

\section{Globalizable partial actions}\label{s:Glob}

Recall the definition of shrinking space.

\begin{Def}\label{d:ShrinKSpace}
	A topological space is a \textit{shrinking space} if every open cover is shrinkable, namely, there is another open cover indexed by the same indexing set, with the property that the closure of each set lies inside the corresponding set in the original cover.
	We say that a space is \textit{countably shrinking} if every countable open cover is shrinkable.
\end{Def}

\begin{Prop}\label{p:DimGlobSp}
	Let $G$ be a countable discrete group, let $X$ be a \LCH second countable space, and let $\theta$ be a partial action of $G$ on $X$ that admits a Hausdorff globalization, $\alpha$ acting on $Y$. Then $\dim(Y)= \dim(X)$.
\end{Prop}

\begin{proof}
	By \autoref{r:PropTopGolob}, we know that $Y$ is a \LCH second countable space.
	Moreover, $Y=\bigcup_{g\in G} \alpha_g(X)$, so $Y$ is a countable union of open subsets, each of them homeomorphic to $X$.
	Dowker's Theorem \cite[Theorem~2]{Dowker} implies that a normal space is countably shrinking if and only if it is countably paracompact. As a metrizable space, $Y$ has these properties. Thus, the above cover is shrinkable and \autoref{t:DimF-Sig} implies that $Y$ has covering dimension equals to the covering dimension of $X$.
\end{proof}

\begin{Cor}\label{c:ResGlob}
	Let $G$, $\theta$, and $X$ be as in \autoref{p:DimGlobSp}.\\
	If $G=\mathbb{Z}$, then
	$$\dimnuc(C_0(X)\rtimes_{\theta}G)\leq 2\dim(X)^2+6\dim(X)+4.$$
	If $G$ is a finite cyclic group, then $$\dimnuc(C_0(X)\rtimes_{\theta}G)\leq \dr(C_0(X)\rtimes_{\theta}G) \leq \dim(X).$$
\end{Cor}

\begin{proof}
	Both results follow from \cite[Corollary~5.4, Theorem~3.4]{HW} and\\ \autoref{p:DimGlobSp}, combined with \autoref{t:MorEqGlob} and \autoref{p:PerPropDimNucDr}(3).
\end{proof}

The following lemma will help us when considering direct limits of partial automorphisms.

\begin{lma}\label{lma:DirectLimitPar}
	Let $X$ be a locally compact Hausdorff space, and let $\theta: U\to V$ be a partial automorphism on $X$. Assume that there exists an increasing sequence $(U_k)_{k=1}^{\infty}$ of open subsets in $X$ such that $U=\bigcup_{k=1}^{\infty} U_k$.\\ Let $\theta^{(k)}: U_k\to \theta(U_k)$ be the restriction of $\theta$ to $U_k$. Then \[C(X)\rtimes_\theta \mathbb{Z}= \varinjlim_k C(X)\rtimes_{\theta^{(k)}}\mathbb{Z}.\]
\end{lma}

\begin{proof}
	Let $\{D^{(k)}_{n}\}_{n\in\mathbb{Z}}$ denote the domains of the $\mathbb{Z}$-partial action generated by $\theta^{(k)}$.\\
	It follows that the Fell-bundle associated to $\theta^{(k)}$, $\{ C_0(D_n^{(k)})\delta_n \}_{n\in\mathbb{Z}}$, is a Fell-sub-bundle of the Fell-bundle associated to $\theta^{(m)}$, $\{ C_0(D_n^{(m)})\delta_n \}_{n\in\mathbb{Z}}$, whenever $k\leq m$.
	Indeed, by \cite[~Definition 21.5]{Exel}, we need to show that $C_0(D_n^{(k)})\delta_n$ is a closed subspace of $C_0(D_n^{(m)})\delta_n$, for every $n$.
	Clearly, it is enough to check that $D_n^{(k)}$ is an open subset of $D_n^{(m)}$, for every $n$.\\ 
	As we are dealing with domains of partial automorphisms, we only need to check it for $n=-1$. However,  $D_{-1}^{(k)}=U_k\subseteq U_m=D_{-1}^{(m)}$.
	Proposition 21.7 in \cite{Exel} implies that for $k\leq m$, $C_0(X)\rtimes_{\theta^{(k)}}\mathbb{Z}$ is a $C^*$-subalgebra of $C_0(X)\rtimes_{\theta^{(m)}}\mathbb{Z}$.\\ 
	Since the domains of $\theta$ are the increasing union of the domains of the $\theta^{(k)}$s, we have $C_0(X)\rtimes_{\theta}\mathbb{Z}=\varinjlim_k{C_0(X)\rtimes_{\theta^{(k)}}\mathbb{Z}}$.
\end{proof}

Next, we show that a partial automorphism on a zero-dimensional, \LCH space is a direct limit of globalizable partial automorphisms. We can then apply \autoref{c:ResGlob}, and bound the nuclear dimension of the associated crossed product.

We start by recalling the proof of the following fact.

\begin{Prop}\label{p:ZerOpenSigClop}
	Let $X$ be a zero-dimensional metrizable space. Then every open set in $X$ is $\sigma$-clopen (i.e. a countable union of clopen sets).
\end{Prop}

\begin{proof}
	Let $U$ be an open subset in $X$. For each $n$, set $U_n\coloneqq\bigcup_{x\in X\setminus U}\mathrm{B}_{1/n}(x).$
	Then $X\setminus U_n$ and $X\setminus U$ are closed disjoint sets, and therefore have disjoint clopen neighborhoods. In particular, for each $n$, there exists 
	a clopen set $V_n$ such that $X\setminus U_n\subseteq V_n$ and $V_n\cap (X\setminus U)=\emptyset$. Clearly, $\bigcup_{n\in\mathbb{N}}V_n=U$, as desired.
\end{proof}

\begin{Cor}\label{c:ZerOpSigClopHomeo}
	Let $X$ be a zero-dimensional metrizable space, and let $\theta: U\to V$ be a homeomorphism between two open sets in $X$. Then there exists an increasing sequence $(U_n)_{n=1}^{\infty}$ of clopen subsets in $X$, such that $U=\bigcup_{n=1}^{\infty}U_n$, and $\theta(U_n)$ is clopen in $X$ for every $n\in\mathbb{N}$.
\end{Cor}

\begin{proof}
	Apply \autoref{p:ZerOpenSigClop} in order to view $U=\bigcup_{i=1}^{\infty} W_i$, where each $W_i$ is clopen in $X$. For each $i$, apply again \autoref{p:ZerOpenSigClop} for the open set $\theta(W_i)$, and view $\theta(W_i)=\bigcup_{j=1}^{\infty}V_i^{(j)}$, where each $V_i^{(j)}$ is clopen in $X$. The countable family $\mathcal{U}\coloneqq \{\theta^{-1}(V_i^{(j)})\colon i,j\in\mathbb{N} \}$ is a clopen cover of $U$. Enumerate the sets in $\mathcal{U}$ by $(U_k')_{k=1}^{\infty}$. For $n\in\mathbb{N}$, set $U_n\coloneqq \bigcup_{k=1}^{n}U_k'$. Then $(U_n)_{n=1}^{\infty}$ is an increasing sequence of clopen subsets in $X$, such that $U=\bigcup_{n=1}^{\infty} U_n$, and $\theta(U_n)$ is clopen in $X$. 
\end{proof}

\begin{Thm}\label{t:ResZeroDim}
	Let $X$ be a zero-dimensional, \LCH second countable space. Assume that $\theta$ is a partial automorphism of $X$. Then $$\dimnuc(C_0(X)\rtimes_{\theta}\mathbb{Z})\leq 4.$$
\end{Thm}

\begin{proof}
	Let $\theta: U\to V$ be a homeomorphism between two open sets of $X$. \\
	By \autoref{c:ZerOpSigClopHomeo}, we can write $U=\bigcup_{k\in\mathbb{N}}U_k$ as an increasing union of clopen sets in $X$ such that $\theta(U_k)$ is clopen in $X$ for every $k\in\mathbb{N}$. Consider the partial automorphism on $X$ generated by the homeomorphism $\theta^{(k)}:U_k\to \theta(U_k)$, where $\theta^{(k)}$ denotes the restriction $\theta$ to $U_k$. Notice that the domains $\{D^{(k)}_{n}\}_{n\in\mathbb{Z}}$ of $\theta^{(k)}$ are clopen.
	By \autoref{lma:DirectLimitPar}, we have $C_0(X)\rtimes_{\theta}\mathbb{Z}=\varinjlim_k{C_0(X)\rtimes_{\theta^{(k)}}\mathbb{Z}}$.\\
	Finally, by \autoref{r:PropTopGolob}, each $\theta^{(k)}$ admits a Hausdorff globalization, so $$\dimnuc(C_0(X)\rtimes_{\theta^{(k)}}\mathbb{Z})\leq 4,$$
	by \autoref{c:ResGlob}. Therefore, $$\dimnuc(C_0(X)\rtimes_{\theta}\mathbb{Z})\leq 4,$$
	by \autoref{p:PerPropDimNucDr}(1).
\end{proof}

We remark that the above method works whenever a partial action can be viewed as a direct limit of globalizable partial actions.

\section{Partial actions of finite groups}\label{s:FinGrp}

We decompose finite group partial actions into globalizable subsystems, which have subhomogeneous crossed products. This will allow us to bound the nuclear dimension, independently of the number $|G|$.

\begin{Def}
	Let $n\in\mathbb{N}$. A $C^*$-algebras $A$ is called \emph{$n$-subhomogeneous} if all its irreducible representations are of dimension at most $n$.
\end{Def}

It is a theorem of Blackadar that $A$ is $n$-subhomogeneous if and only if $A$ embeds into $M_n(C_0(X))$ for some locally compact Hausdorff space $X$ (see \cite[Proposition~3.4.2]{RorYBook}). In particular, every $C^*$-subalgebra of an $n$-subhomogeneous $C^*$-algebra is again $n$-subhomogeneous.

For the \emph{primitive ideal space} of a $C^*$-algebra and its topological properties, we refer to \cite[Chapter~II.6.5]{Black2}. For $k\in\mathbb{N}$, denote by $\mathrm{Prim}_k(A)$ the set of kernels of $k$-dimensional irreducible representations. We equip $\mathrm{Prim}_k(A)\subseteq \mathrm{Prim}(A)$ with the induced subspace topology.\\

Recall the following theorem by Winter.

\begin{Thm}\cite[Theorem~1.6]{DrSH}  \label{t:WinDrSH}
	Let $A$ be a separable, subhomogeneous $C^*$-algebra. Then $$\dr(A)=\max_k\{\dim \ \Prim_k(A)\}$$
\end{Thm}

We present an easy corollary of this theorem, which will be useful for us.

\begin{lma}\label{l:DrExtSH}
	Let $0\to A\to E\to B\to 0$ be a short-exact sequence of $C^*$-algebras. Assume that $E$ is separable and unital, and that $A,B$ are $n$-subhomogeneous.
	Then $E$ is $n$-subhomogeneous and 
	$$\dr(E)= \max\{\dr(A),\dr(B)\}$$
\end{lma}

\begin{proof}
	As $E^{**}=A^{**}\oplus B^{**}$, we see that $E$ is $n$-subhomogeneous. The inequality $$\dr(E)\geq  \max\{\dr(A),\dr(B)\}$$ follows from \autoref{p:PerPropDimNucDr}. We prove the converse inequality. Since $E$ is separable, $\Prim_k(E)$ is a \LCH second countable space. By Urysohn's metrization theorem, it is metrizable. Since $E$ is unital and subhomogeneous, we have $\Prim_k(E)=\Prim_k(A)\cup \Prim_k(B)$, where $\Prim_k(B)\subseteq \Prim_k(E)$ is closed and $\Prim_k(A)\subseteq \Prim_k(E)$ is open.
	Since every open set in a metrizable space is an $\mathrm{F}_\sigma$-set, the result follows from \autoref{t:DimF-Sig} and \autoref{t:WinDrSH}.
\end{proof}

\begin{Thm}\label{t:ResCptFinGrp}
	Let $X$ be a compact Hausdorff second countable space, let $G$ be a finite cyclic group and let $\theta=(\{D_g\}_{g\in G},\{\theta_g\}_{g\in G})$ be a topological partial action of $G$ on $X$. Denote the induced partial action on $C(X)$ by $\hat{\theta}$. Then 
	$$\dr(C(X)\rtimes_{\hat{\theta}}G)\leq \dim(X).$$
\end{Thm}

In order to make the following somewhat involved proof more transparent, we briefly explain the case of a topological partial action of a $4$ element group $G=\{g_1=e,g_2,g_3,g_4 \}$.

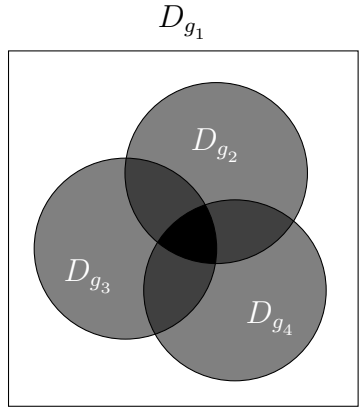
\begin{figure}[h]\label{fig:DecGlob}
	\centering
	\begin{tikzpicture}
	\begin{scope}
	\fill[gray]\firstcircle;
	\fill[gray] \secondcircle;
	\fill[gray] \thirdcircle;
	\end{scope}
	\begin{scope}
	\clip \secondcircle;
	\fill[darkgray] \firstcircle;
	\end{scope}
	\begin{scope}
	\clip \secondcircle;
	\fill[darkgray] \thirdcircle;
	\end{scope}
	\begin{scope}
	\clip \firstcircle;
	\fill[darkgray] \thirdcircle;
	\end{scope}
	\begin{scope}
	\clip \firstcircle;
	\clip \secondcircle;
	\fill[black] \thirdcircle;
	\end{scope}[font=\large]
	\draw \firstcircle node[text=white,above] {$D_{g_{_2}}$};
	\draw \secondcircle node [text=white,below left] {$D_{g_{_3}}$};
	\draw \thirdcircle node [text=white,below right] {$D_{g_{_{4}}}$};
	\draw ([xshift=-0.8em,yshift=-0.8em]current bounding box.south west)
	rectangle
	([xshift=1em,yshift=1em]current bounding box.north east);
	\path (current bounding box.north) node[above]{$D_{g_{_1}}$};
	\end{tikzpicture}
	\caption{Decomposition into globalizable subsystems}
\end{figure}

We decompose inductively our system into invariant subsystems which admit globalizations and apply \autoref{c:ResGlob}. We first consider the restriction of $\theta$ to the intersection of all domains (the black area in \autoref{fig:DecGlob}), there it acts globally. Then, we restrict our system to the darker grey area, which can be checked to be invariant, and the restricted system has clopen domains, thus globalizable by \autoref{r:PropTopGolob}(4). Next, we consider the restriction to the (invariant) light grey area, there all non-trivial domains are disjoint, and therefore clopen. 
Finally, for the restriction to the white area, all non-trivial domains are empty, thus the crossed product is simply the commutative algebra $C(X\setminus(D_{g_2}\cup D_{g_3}\cup D_{g_4}))$.
In this way, we obtain inductively short exact sequences and \autoref{p:PerPropDimNucDr}(1) gives a bound for the nuclear dimension of the crossed product. In order to get a bound independent of the cardinality of the group, we apply \autoref{l:DrExtSH}.

\begin{proof}[Proof of \autoref{t:ResCptFinGrp}]
	Let $n$ denote the cardinality of $G$. We define inductively short exact sequences of $C^*$-algebras
	$$0\to A_k\to E_k\to B_k\to 0, \ \ \ \ \ (0\leq k\leq n-2) $$
	such that the following properties are satisfied.
	\begin{enumerate}
		\item $E_0=C(X)\rtimes_{\hat{\theta}} G$.
		\item $A_k=C_0(X^{(k)})\rtimes_{\hat{\theta}} G$, where $X^{(k)}$ is a subspace of $X$, and $\hat{\theta}|_{X^{(k)}}$ admits a Hausdorff globalization.
		\item $A_k$ is $n$-subhomogeneous and $\dr(A_k)\leq \dim(X)$.
		\item $E_{k+1}=B_k$ separable and unital, for $0\leq k\leq n-3$.
		\item $B_{n-2}$ is a commutative $C^*$-algebra, with $\dr(B_{n-2})\leq \dim(X)$.
	\end{enumerate}
	
	We observe that $(2)\implies (3)$, as by \autoref{t:MorEqGlob}, \cite[Theorem~2.1]{HW} and \autoref{c:ResGlob}, $(2)$ implies that each $A_k$ is a full hereditary subalgebra of an $n$-subhomogeneous $C^*$-algebra with decomposition rank bounded by $\dim(X)$. Notice that once we construct such short exact sequences, the proof is complete, applying \autoref{l:DrExtSH} repeatedly.
	
	Write $G=\{g_1,\ldots,g_n\}$. 
	Set $X^{(0)}=\bigcap_{g\in G} D_g$. Then $\theta|_{X^{(0)}}$ is a global action. Let $k\geq 1$, and assume that we have defined $X^{(0)},\ldots, X^{(k-1)}$ which satisfy property (2) above. We denote by $\theta^{(k)}$ the restriction of $\theta$ to ${X\setminus \bigcup_{i=0}^{k-1}X^{(i)}}$, and by $\{D_g^{(k)}\}_{g\in G}$ the domains of $\theta^{(k)}$. Set $$X^{(k)}\coloneqq \bigcup\limits_{m_1,\ldots,m_k\in \{1,\ldots,n\}}\bigcap\limits_{g\in G\setminus{ \{g_{m_1},\ldots,g_{m_k}\}  } }D_g^{(k)}.$$
	In other words, $X^{(k)}$ consists of intersections of (at least) $n-k$ domains. Notice that intersections of more than $n-k$ domains would belong to $X^{(0)}\cup\ldots \cup X^{(k-1)}$, which was taken out. Therefore, the domains of $\theta^{(k)}$ satisfy the property that any intersection of $n-k+1$ of them is empty, and we actually have
	$$X^{(k)}= \bigsqcup\limits_{m_1,\ldots,m_k\in \{1,\ldots,n\}}\bigcap\limits_{g\in G\setminus{ \{g_{m_1},\ldots,g_{m_k}\}  } }D_g^{(k)}.$$
	This shows that the domains of $\theta|_{X^{(k)}}$ are clopen and it admits a Hausdorff globalization by \autoref{r:PropTopGolob}(4). Moreover, \autoref{o:DomRel} implies that $X^{(k)}$ is an open $\theta^{(k)}$-invariant subset. 
	To this end, we observe that conditions (1)-(5) determine the short exact sequences; for $k\geq 0$ we have
	\[E_k= C\big(X\setminus \bigcup_{i=0}^{k-1}X^{(i)}\big)\rtimes_{\theta^{(k)}} G, \text{ and } \]
	\[B_k=C\big(X\setminus \bigcup_{i=0}^{k}X^{(i)}\big)\rtimes_{\theta^{(k+1)}} G.\]
	The $C^*$-algebras $E_{k+1}=B_{k}$ are unital and separable, since $C(X)$ is separable whenever $X$ is a compact metrizable space.
	To see that $B_{n-2}$ is a commutative $C^*$-algebra, recall that $\theta^{(n-1)}$ has the property that its domains are pairwise disjoint. Thus, its only non-empty domain is the one corresponds to the identity element of the group. Using \autoref{t:DimSub} and \autoref{p:PerPropDimNucDr}(4), we conclude that $\dr(B_{n-2})\leq \dim(X)$.
\end{proof}

\begin{Rmk}\label{r:DesFinCP}
	\mbox{}
	\begin{enumerate}
		\item The assumption that $G$ is cyclic is only needed in order to apply the known results for global systems. In work in progress, Hirshberg and Wu show that crossed products associated to actions of finitely generated virtually nilpotent groups on finite dimensional spaces always have finite nuclear dimension. In particular, they will cover finite groups, and one would be able to obtain a bound also for partial actions by general finite groups, following the same proof as for \autoref{t:ResCptFinGrp}.
		\item With more work, it is possible to give a similar decomposition of the crossed product algebra also in the noncommutative case, and describe explicitly the algebras $\{A_k\}_{k=0}^{n-2}$ that appear in the proof, see \cite{AGG} for more details.
		
	\end{enumerate}
\end{Rmk}

\begin{Cor}\label{c:ResFinGrp}
	Let $\theta=(\{D_g\}_{g\in G}, \{\theta_g\}_{g\in G})$ be a partial action of a finite cyclic group $G$ on a \LCH second countable space $X$. Then 
	$$\dr(C_0(X)\rtimes_{\theta} G)\leq \dim(X).$$
\end{Cor}

\begin{proof}
	The one-point compactification of $X$, $X^{+}$, is a compact Hausdorff second countable space.
	Let $\theta^{+}=(\{E_g\}_{g\in G}, \{\theta_g^{+}\}_{g\in G})$ be the partial action of $G$ on $X^+$, defined by: \\
	\[E_g=D_g \text{ \ for all \ } g\neq 1, \  E_1 = X^{+}; \  \ \theta_g^+=\theta_g \text{ \ for all \ } g\neq 1, \ \theta_1^+=\id_{X^+}.\]
	Then $X$ is an open $\theta^+$-invariant subspace. The restriction of $\theta^{+}$ to $X$ is exactly $\theta$. Thus, by \autoref{p:InvSetS.e.s}, $C_0(X)\rtimes_{\theta} G$ is an ideal inside $C(X^+)\rtimes_{\theta^{+}}G$. By \autoref{t:ResCptFinGrp}: $$\dr(C_0(X)\rtimes_{\theta} G)\leq \dr(C(X^+)\rtimes_{\theta^{+}}G)\leq \dim(X^+)=\dim(X).$$
\end{proof}

\section{Partial automorphisms with supported domains}
\label{s:SupDom}	
In this section we show that crossed products by partial automorphisms with \textit{finitely supported domains}, namely
$$\ldots=\emptyset=\emptyset=\ldots=\emptyset\subseteq D_{n}\subseteq\ldots\subseteq D_1\subseteq X,$$
and (automatically)
$$\ldots=\emptyset=\emptyset=\ldots=\emptyset\subseteq D_{-n}\subseteq\ldots\subseteq D_{-1}\subseteq X,$$
for some $n\in\mathbb{N}$, are isomorphic to crossed products by partial actions of finite cyclic groups. Thus, we can apply the results from the previous section. As a consequence, we can deal with partial automorphisms with \textit{uniformly bounded orbits}, that we define below.

\begin{Claim}\label{cl:FinSupDom}
	Let $\theta$ be a partial automorphism on a \LCH space $X$. If $\theta$ has finitely supported domains, and $n\in\mathbb{N}$ is such that $D_{n+1}=\emptyset$, then
	there exists a partial action $\eta$ of $\mathbb{Z}/(2n+1)\mathbb{Z}$ on $X$, such that
	$$C_0(X)\rtimes_{\theta} \mathbb{Z}\cong C_0(X)\rtimes_{\eta} \mathbb{Z}/(2n+1)\mathbb{Z}.$$
\end{Claim}

\begin{proof}
	Denote $\mathbb{Z}/(2n+1)\mathbb{Z}=\{\overline{-n}, \ldots \overline{-1} ,\overline{0}, \overline{1},\ldots, \overline{n} \}$.
	In order to define $\eta$, we first give the domains $\{E_{\overline{k}}\}_{k=-n}^{n}$ and the maps $\{\eta_{\overline{k}}\}_{k=-n}^{n}$. For $ k\in\{-n,\ldots, n\}$, set
	$$E_{\overline{k}}=D_k,  \text{ \ \ \ and \ \ \  } \eta_{\overline{k}}=\theta_k$$
	We have to check that $\eta$ is a well-defined partial action of $\mathbb{Z}/(2n+1)\mathbb{Z}$. Namely, 
	$$\eta_{\overline{m}}\circ\eta_{\overline{k}}\subseteq \eta_{\overline{m+k}}, $$
	for all $m,k\in \{-n,\ldots, n\}$. 
	We separate to cases:
	\begin{enumerate}
		\item If $-n\leq m+k\leq n$, then
		$$\eta_{\overline{m}}\circ\eta_{\overline{k}}=\theta_m\circ\theta_k\subseteq \theta_{m+k}=\eta_{\overline{m+k}}, $$
		because $\theta$ is a partial action of $\mathbb{Z}$.
		
		\item If $n+1\leq m+k\leq 2n$, then $m+k-(2n+1)\in\{-n,\ldots,-1\}$. Moreover, $\theta_{m+k}=\emptyset$, since its domain is empty by assumption. Therefore,
		$$\eta_{\overline{m}}\circ\eta_{\overline{k}}=\theta_m\circ\theta_k\subseteq\theta_{m+k}=\emptyset\subseteq \theta_{m+k-(2n+1)}=\eta_{\overline{m+k}}.  $$ 
		
		\item If $-2n\leq m+k\leq -(n+1)$, then $m+k+(2n+1)\in\{1,\ldots,n\}$. Moreover, $\theta_{m+k}=\emptyset$, since its domain is empty by assumption. Therefore,
		$$\eta_{\overline{m}}\circ\eta_{\overline{k}}=\theta_m\circ\theta_k\subseteq\theta_{m+k}=\emptyset\subseteq \theta_{m+k+(2n+1)}=\eta_{\overline{m+k}}.  $$ 
	\end{enumerate}
	
	(1)-(3) cover all possible cases. It can be easily checked that
	$$\phi\colon C_0(X)\rtimes_{\eta}\mathbb{Z}/(2n+1)\mathbb{Z}\to C_0(X)\rtimes_{\theta} \mathbb{Z}$$
	$$\phi(\sum\limits_{k=-n}^{n}f_{\overline{k}}\delta_{\overline{k}})=\sum\limits_{k=-n}^{n} f_k\delta_k, \text{ \ where \ } f_{\overline{k}}\in C_0(E_{\overline{k}}),$$
	defines a natural isomorphism between the crossed product algebras.
\end{proof}

Combining with the results of \autoref{s:FinGrp}, we get an immediate corollary,

\begin{Cor}\label{c:ResFinSupDom}
	Let $\theta$ be a partial automorphism on a \LCH second countable space $X$. If $\theta$ has finitely supported domains, then
	\[\dr(C_0(X)\rtimes_{\hat{\theta}}\mathbb{Z})\leq \dim(X). \]
\end{Cor}

Next, we apply our results to partial automorphisms with uniformly bounded domains (defined below) by showing that they can be realized as an extension of a periodic global action by a partial action of a finite cyclic group. 

\begin{Def}\label{d:UniBndOrb}
	Let $\theta$ be a partial action of a discrete group $G$ on a space $X$. We say that $\theta$ has \textit{uniformly bounded orbits} if there exists $N\in\mathbb{N}$ such that for every $x\in X$: 
	$$\left| \mathrm{Orb}_{\theta}(x) \right|= \left|\{\theta_g(x): x\in D_{g^{-1}} \}\right|\leq N.$$
\end{Def}

We remark that global actions of $\mathbb{Z}$ have bounded orbits if and only if they are periodic.

\begin{Obs}\label{o:UniBndDom}
	Let $\theta$ be a partial automorphism on a space $X$. If $\theta$ has uniformly bounded orbits, then there exists $m\in\mathbb{N}$ such that
	$$D_n=D_m \text{ for all } |n|\geq m$$.
\end{Obs}

\begin{proof}
	Let $N$ be such that for every $x\in X$, $\left|\{\theta_n(x)\colon x\in D_{-n} \}\right|\leq N$. We claim that $\theta_{N!}=\id_{D_{-N!}}$.  In particular, $D_{N!}=D_{-N!}$.\\ 
	Let $x\in D_{-N!}$. Since $N\leq N!$, \autoref{r:IncDom} implies that $x\in D_{-N}$. Now, $$\{x,\theta_1(x),\ldots,\theta_N(x)\}\subseteq \mathrm{Orb}_{\theta}(x).$$
	By assumption, there exist $l,j\in \{0,\ldots,N\}$ such that $l<j$ and $\theta_j(x)=\theta_l(x)$. Therefore, $\theta_{j-l}(x)=x$. In other words, there exists $p\in \{1,\ldots,N\}$ so that $\theta_p(x)=x$. \\
	The claim follows, since $\theta_{N!}(x)=\underbrace{\theta_p\circ\theta_p\circ\ldots\circ \theta_p}_{\frac{N!}{p}\in\mathbb{N}}(x)=x.$\\
	Therefore, for all $k\in\mathbb{N}$, $D_{-kN!}=\mathrm{Dom}(\theta_{kN!})=\mathrm{Dom}(\theta_{N!}\circ\ldots\circ \theta_{N!})=D_{-N!}$. Similarly, by considering the ranges, $D_{kN!}=D_{N!}$ for all $k\in\mathbb{N}$.\\
	\autoref{r:IncDom} implies that $D_n=D_{N!}$ for all $n\in \mathbb{Z}$ such that $\left|n\right|\geq N!$.
\end{proof}

\begin{Cor}\label{c:ResUniBnd}
	Let $\theta$ be a partial automorphism on a \LCH second-countable space $X$, with uniformly bounded orbits. Then 
	$$\dr(C_0(X)\rtimes_{\theta} \mathbb{Z})\leq \dim(X)+1.$$
\end{Cor}

\begin{proof}
	Following the same trick done in \autoref{c:ResFinGrp}, we may assume without loss of generality that $X$ is compact.
	Using \autoref{o:UniBndDom}, we find $m$ such that the domains of $\theta$ are of the following form,\\
	$$\ldots=D_m=D_m=\ldots=D_m\subseteq D_{m-1}\subseteq\ldots\subseteq D_1\subseteq X;$$
	and
	$$\ldots=D_m=D_m=\ldots=D_m\subseteq D_{-(m-1)}\subseteq\ldots\subseteq D_{-1}\subseteq X.$$
	As $D_m$ is an open invariant subset, we have a short exact sequence
	$$0\to C_0(D_m)\rtimes \mathbb{Z}\to C(X)\rtimes \mathbb{Z}\to C(X\setminus D_m)\rtimes \mathbb{Z}\to 0.$$
	
	The proof of \autoref{o:UniBndDom} shows that the restriction of $\theta$ to $D_m$ is a global periodic action.
	Thus, $C_0(D_m)\rtimes \mathbb{Z}$ is subhomogeneous, and $\dr(C_0(D_m)\rtimes \mathbb{Z})\leq \dim(X)+1$ (cf. \cite[proposition~2.1]{HW}).
	Moreover, if we denote the domains of $\theta|_{X\setminus D_m}$ by $\{E_k\}_{k\in\mathbb{Z}}$, we see that 
	$E_m=\emptyset$. So, by \autoref{c:ResFinSupDom}, $C(X\setminus D_m)\rtimes \mathbb{Z}$ is isomorphic to a crossed product by a partial action of a finite cyclic group, and therefore, it is subhomogeneous, and has decomposition rank bounded by the dimension of $X$ (see \autoref{t:ResCptFinGrp}).
	Finally, we can apply \autoref{l:DrExtSH} to get the desired bound.
\end{proof}

\section{Partial automorphisms with dense domains}\label{s:DenDom}

The main theorem of this section implies that in order to bound the nuclear dimension of a crossed product associated to a partial automorphism, we can restrict ourselves to the understanding of partial automorphisms with domains that are dense in the base space. We first need the following lemma.

\begin{lma}\label{l:InvInterDom}
	Let $\theta=(\{D_n\}_{n\in\mathbb{Z}},\{\theta_n\}_{n\in\mathbb{Z}})$ be a partial automorphism on a topological space $X$. Then
	\[Y_+\coloneqq \bigcap_{n\geq 0}\overline{D_n} \ \ \text{ and } \ \  Y_-\coloneqq \bigcap_{n\leq 0}\overline{D_n}  \]
	are closed, $\theta$-invariant subsets of $X$.
\end{lma}

\begin{proof}
	We denote also the generating homeomorphism by $\theta:D_{-1}\to D_1$. Let $n\in\mathbb{Z}$, and let $x\in \overline{D_n}\cap D_{-1}$ be given. Choose a net $(x_\lambda)_{\lambda\in \Lambda}$ such that $x_\lambda\in D_n$ for all $\lambda\in\Lambda$, and $\lim_\lambda x_\lambda=x$. Since $D_{-1}$ is open, there exists $\lambda_0\in \Lambda$ such that $x_\lambda\in D_{-1}$ for all $\lambda\geq\lambda_0$. By continuity of $\theta$, we have
	$\theta(x)=\lim_{\lambda\geq\lambda_0}\theta(x_\lambda)$. However
	\[\theta(x_\lambda)\in \theta(D_n\cap D_{-1})\subseteq D_{n+1} , \]
	for all $\lambda\geq\lambda_0$, and thus $\theta(x)\in \overline{D_{n+1}}$. As $x$ was an arbitrary element, we get that $\theta(\overline{D_n}\cap D_{-1})\subseteq \overline{D_{n+1}}$ for all $n\in\mathbb{Z}$. Therefore,
	\[\theta(Y_+\cap D_{-1})=\bigcap_{n\geq 0}\theta(\overline{D_n}\cap D_{-1})\subseteq \bigcap_{n\geq 0} \overline{D_{n+1}}=Y_+ , \text{ and } \]
	\[ \theta(Y_-\cap D_{-1})=\bigcap_{n\leq 0}\theta(\overline{D_n}\cap D_{-1})\subseteq \bigcap_{n\leq 0} \overline{D_{n+1}}\subseteq Y_- . \]
	Similarly, $Y_+$ and $Y_{-}$ are invariant under $\theta^{-1}$. The remark in \autoref{d:InvSetMinParAct} implies that these sets are invariant for the generated partial action.
\end{proof}

\begin{Thm}\label{t:ResIntDom}
	Let $\theta=(\{D_n\}_{n\in\mathbb{Z}},\{\theta_n\}_{n\in\mathbb{Z}})$ be a partial automorphism on a \LCH second countable space $X$. Then
	\[\dimnuc(C_0(X)\rtimes_{\theta}\mathbb{Z})\leq \dimnuc(C_0\big(\bigcap_{n\in\mathbb{Z}}\overline{D_n}\big)\rtimes_{\theta}\mathbb{Z})+2\dim(X)+2. \]
\end{Thm}

\begin{proof}
	We denote also the generating homeomorphism by $\theta:D_{-1}\to D_1$. Set 
	\[Y_-\coloneqq \bigcap_{n\geq 0}\overline{D_{-n}}  \]
	By \autoref{l:InvInterDom}, $Y_-$ is a closed $\theta$-invariant subset of $X$. 
	We obtain a short exact sequence of $C^*$-algebras,
	\[0\to C_0(X\setminus Y_-)\rtimes_{\theta|_{X\setminus Y_-}}\mathbb{Z}\to C_0(X)\rtimes_{\theta}\mathbb{Z}\to C_0(Y_-)\rtimes_{\theta|_{Y_-}}\mathbb{Z}\to 0.\]
	Set $E_n\coloneqq D_n\cap (X\setminus Y_-)$ for all $n\in\mathbb{Z}$. These are the domains of $\theta|_{X\setminus Y_-}$. This restriction is a partial automorphism generated by $\theta: E_{-1}\to E_1$. Now, observe that $\{E_{-1}\setminus \overline{E_{-n}}\}_{n\geq 0}$, where the closures are taken inside $E_0$, is an increasing sequence of open sets in $E_0$, such that $E_{-1}=\bigcup_{n\geq 0}E_{-1}\setminus \overline{E_{-n}}$. Indeed,
	\begin{align*}
	\bigcup_{n\geq 0}E_{-1}\setminus \overline{E_{-n}}&=E_{-1}\setminus \bigcap_{n\geq 0}\overline{E_{-n}}\\
	&=E_{-1}\setminus \left(\bigcap_{n\geq 0}\overline{D_{-n}}\cap (X\setminus Y_-)\right)\\
	&=E_{-1}\setminus (Y_-\cap (X\setminus Y_-))=E_{-1}.
	\end{align*}
	Notice that for each $n\leq 0$, the partial automorphism generated by the restriction of $\theta$ to $E_{-1}\setminus \overline{E_{n}}$ has finitely supported domains. Indeed, its $n$-th domain must be a subset of $E_{-1}\setminus \overline{E_{n}}$ (see \autoref{r:IncDom}), and also a subset of the $n$-th domain of the original $\theta: E_{-1}\to E_1$, namely $E_{n}$. By \autoref{lma:DirectLimitPar}, if we denote $\theta^{(n)}\coloneqq\theta|_{E_{-1}\setminus \overline{E_{-n}}}$ for all $n\geq 0$, we have
	\[C_0(X\setminus Y_-)\rtimes_{\theta} \mathbb{Z}=\varinjlim_n C_0(X\setminus Y_-)\rtimes_{\theta^{(n)}} \mathbb{Z}.\]
	Combining \autoref{p:PerPropDimNucDr} with \autoref{c:ResFinSupDom}, we obtain \[\dimnuc(C_0(X)\rtimes_{\theta}\mathbb{Z})\leq \dimnuc(C_0(Y_-)\rtimes_{\theta}\mathbb{Z})+\dim(X)+1.\]
	We now deal with the system induced by the restriction of $\theta$ to $Y_-$.
	Set $F_n\coloneqq D_n\cap Y_-$ for all $n\in \mathbb{Z}$. $\{F_n\}_{n\in\mathbb{Z}}$ are the domains of $\theta|_{Y_-}$. This restriction is a partial automorphism generated by $\theta: F_{-1}\to F_1$. Set
	\[Y_+\coloneqq \bigcap_{n\geq 0} \overline{F_n}, \]
	where the closures are taken inside $Y_-$. By \autoref{l:InvInterDom}, $Y_+$ is a closed invariant set.
	Thus, we obtain a short exact sequence of $C^*$-algebras:
	\[0\to C_0(Y_-\setminus Y_+)\rtimes_{\theta}\mathbb{Z}\to C_0(Y_-)\rtimes_{\theta}\mathbb{Z}\to C_0(Y_+)\rtimes_{\theta}\mathbb{Z}\to 0.\]
	Essentially the same argument as above gives that the left hand sided partial system is a direct limit of partial automorphisms with finitely supported domains, and thus
	\[\dimnuc(C_0(Y_-\setminus Y_+)\rtimes_\theta\mathbb{Z})\leq \dim(X).\]
	As $Y_+=\bigcap_{n\in\mathbb{Z}}\overline{D_n}$, we get by the estimate of nuclear dimension to extensions,
	\[\dimnuc(C_0(Y_-)\rtimes_{\theta}\mathbb{Z})\leq \dimnuc(C_0\big(\bigcap_{n\in\mathbb{Z}}\overline{D_n}\big)\rtimes_{\theta}\mathbb{Z})+ \dim(X)+1. \]
	We conclude that
	\[\dimnuc(C_0(X)\rtimes_{\theta}\mathbb{Z})\leq \dimnuc(C_0\big(\bigcap_{n\in\mathbb{Z}}\overline{D_n}\big)\rtimes_{\theta}\mathbb{Z})+2\dim(X)+2,\]
	as required.
\end{proof}

\begin{Cor}\label{c:IntCases}
	Let $\theta=(\{D_n\}_{n\in\mathbb{Z}},\{\theta_n\}_{n\in\mathbb{Z}})$ be a partial automorphism on a \LCH second countable space $X$. Then
	\begin{enumerate}
		\item If $\theta$ restricts to a global action on $\bigcap_{n\in\mathbb{Z}}\overline{D_n}$, then
		\[ \dimnuc(C_0(X)\rtimes_{\theta}\mathbb{Z})\leq 2\dim(X)^2+8\dim(X)+6.  \] 
		\item If $\dim(\bigcap_{n\in\mathbb{Z}}\overline{D_n})=0$, then
		\[ \dimnuc(C_0(X)\rtimes_{\theta}\mathbb{Z})\leq 2\dim(X)+6.  \] 
	\end{enumerate}
\end{Cor}

\begin{proof}
	These inequalities follow immediately from \autoref{t:ResIntDom}, \autoref{t:HirWu}, and \autoref{t:ResZeroDim}.
\end{proof}

Next, consider a concrete example from \cite{Exel2}.
Denote by $\mathbb{N}^+=\mathbb{N}\cup\{\infty\}$ the one-point compactification of $\mathbb{N}$. Let $\theta$ be the partial automorphism on $\mathbb{N}^+$ generated by the homeomorphism $\theta: \mathbb{N}^+\setminus \{1\}\to \mathbb{N}^+$, given by $\theta(n)=n-1$ for $n\in\mathbb{N}\setminus\{1\}$, and $\theta(\infty)=\infty$. The domains of the induced topological partial action of $\mathbb{Z}$ on $\mathbb{N}^+$ are given by
\[ D_n=\begin{cases}
\mathbb{N}^+ &,\text{ if } n\geq0\\
\mathbb{N}^+\setminus \{1,2,\ldots, |n|\} &,\text{ if } n<0
\end{cases}  \]
One can show that
$C(\mathbb{N}^+)\rtimes_{\theta}\mathbb{Z}=\mathcal{T}$, where $\mathcal{T}$ denotes the Toeplitz algebra. In this case, both conditions of \autoref{c:IntCases} are trivially satisfied. Moreover, note that the restriction of $\theta$ to $\bigcap_{n\in\mathbb{Z}}\overline{D_n}=\{\infty\}$ is the trivial action, and one can show that $C_0(\mathbb{N})\rtimes_{\theta}\mathbb{Z}\cong K(l^2(\mathbb{N}))$. Therefore, the extension
\[0\to C_0\big(\mathbb{N}^+\setminus \bigcap_{n\in\mathbb{Z}}\overline{D_n}\big)\rtimes_{\theta}\mathbb{Z}\to C(\mathbb{N}^+)\rtimes_{\theta}\mathbb{Z}\to C\big(\bigcap_{n\in\mathbb{Z}}\overline{D_n}\big)\rtimes_{\theta}\mathbb{Z}\to 0 \ , \]
considered in the proof of \autoref{t:ResIntDom}, is exactly the Toeplitz extension
\[0\to K(l^2(\mathbb{N}))\to \mathcal{T}\to C(\mathbb{T})\to 0. \]
However, the bound for the nuclear dimension of $\mathcal{T}$ obtained this way is not optimal; it was shown in \cite{BW} that $\dimnuc(\mathcal{T})=1$.

\section{Minimal partial automorphisms}\label{s:MinParAut}

In this section we bound the nuclear dimension of crossed products associated to minimal partial automorphisms. We first point out that, using different methods, a better bound can be achieved when dealing with extendable homeomorphisms.

\begin{Thm}\label{t:ResExtMinParAut}
	Let $X$ be an infinite, compact, metrizable, finite-dimensional space with a minimal partial action generated by a partial automorphism $\theta:D_{-1}\to D_1$, with $D_1\subsetneq X$. Assume that $\theta$ extends to a self-homeomorphism $\alpha: X\to X$. Then
	\[\dr(C(X)\rtimes_{\theta}\mathbb{Z})\leq \dim(X). \]
\end{Thm}

\begin{proof}
	Notice that $\alpha$ is a minimal homeomorphism on $X$, since every $\alpha$-invariant set, $Z\subseteq X$, is also $\theta$-invariant. Indeed, given $n\in\mathbb{Z}$, we have
	\[\theta_{n}(D_{-n}\cap Z)=\alpha^{n}(D_{-n}\cap Z)\subseteq \alpha^{n}(Z)\subseteq Z. \]
	
	For a closed subset $Y\subseteq X$, we denote 
	$$A_Y\coloneqq C^*(C(X), C_0(X\setminus Y)u)\subseteq C(X)\rtimes_{\alpha} \mathbb{Z}.$$ 
	Let $Y\coloneqq X\setminus D_1$. By assumption, $Y$ is a non-empty closed subset in $X$. Thus, there exists a decreasing sequence $(Y_m)_{m=1}^{\infty}$ of closed subsets in $X$ such that $\mathrm{int}(Y_m)\neq \emptyset$ and $Y=\bigcap_{m=1}^{\infty} Y_m$. Combining \cite[Section~3]{LP} with \cite[Theorem~1.6]{DrSH}, it follows that each $A_{Y_m}$ is a recursive subhomogeneous $C^*$-algebra with $\dr(A_{Y_m})\leq \dim(X)$. Observe that
	$$C(X)\rtimes_{\theta}\mathbb{Z}=A_Y=\varinjlim_m A_{Y_m}. $$
	The result now follows by \autoref{p:PerPropDimNucDr}(3).
\end{proof}

\begin{Rmk}
	The assumption that $D_1\subsetneq X$ (or, $D_{-1}\subsetneq X$) in \autoref{t:ResExtMinParAut} is necessary. For example, if $\theta$ is a global action on a Cantor set $X$, then 
	\[\dr(C(X)\rtimes_{\theta}\mathbb{Z})\leq \dim(X)=0\]
	would imply that a crossed product of the unital $AF$-algebra, $C(X)$, by $\mathbb{Z}$ is an $AF$ algebra. This leads to a contradiction, see \cite[Exercise~10.11.1]{Black}.
\end{Rmk}

We move to the general case, starting with an easy observation.

\begin{Obs}\label{o:ClosInvSets}
	Let $\theta: D_{-1}\to D_1$ be a partial automorphism on a space $X$. Then the closure of any $\theta$-invariant set is again $\theta$-invariant.
\end{Obs}

Since the complement of a $\theta$-invariant set is always $\theta$-invariant, a partial automorphism is minimal if and only if $X$ and $\emptyset$ are the only closed $\theta$-invariant subsets.

\begin{lma}\label{l:ClosOfDomCont}
	Let $\theta=(\{D_n\}_{n\in\mathbb{Z}}, \{\theta_n\}_{n\in\mathbb{Z}})$ be a partial action of $\mathbb{Z}$ on a topological space $X$, generated by a partial homeomorphism $\theta: D_{-1}\to D_1$. Let $E_{-1}$ be an open subset of $D_{-1}$, and $\eta\coloneqq \theta|_{E_{-1}}$. Denote by $\eta=(\{E_n\}_{n\in\mathbb{Z}}, \{\eta_n\}_{n\in\mathbb{Z}})$ the generated partial action. Assume that $\overline{E_{-1}}\subseteq D_{-1}$ and $\overline{E_1}\subseteq D_1$. Then $\overline{E_n}\subseteq D_n$, for all $n\in\mathbb{Z}$.
\end{lma}

\begin{proof}
	It is enough to show $\overline{E_n}\subseteq D_n$, for all $n\geq 0$. We prove the claim inductively. Let $n\geq 2$ be given, and assume that the claim holds for $n-1$. Then
	\[\overline{E_n}=\overline{\eta(E_{n-1}\cap E_{-1})}=\theta(\overline{E_{n-1}\cap E_{-1}})\subseteq \theta(\overline{E_{n-1}}\cap \overline{E_{-1}})\subseteq\theta(D_{n-1}\cap D_{-1})=D_n, \]
	where at the second step we use that $\overline{E_{-1}}\subseteq D_{-1}$ and $\overline{\theta(E_{-1})}\subseteq D_1$, and at the fourth step we use the inductive assumption.
\end{proof}

\begin{Thm}\label{t:MinimalAut}
	Let $\theta=(\{D_n\}_{n\in\mathbb{Z}}, \{\theta_n \}_{n\in\mathbb{Z}})$ be a minimal partial automorphism on a \LCH second countable space $X$, with $D_{-1}\subsetneq X$. Then
	\[\dimnuc(C_0(X)\rtimes_{\theta} \mathbb{Z})\leq 2\dim(X)+6. \]
\end{Thm}

\begin{proof}
	As observed in \autoref{s:DenDom}, $\bigcap_{n\in\mathbb{Z}} \overline{D_n}$ is a closed $\theta$-invariant set. By minimality, it is either the whole space $X$, or the empty set. In the latter case, the desired bound follows from \autoref{c:IntCases}(2). In the first case, we show that $\theta$ is a direct limit of partial automorphisms on $X$ satisfying the condition in \autoref{c:IntCases}(2). Using the properties of $X$, we can view $D_{-1}=\bigcup\limits_{k=1}^{\infty}U_k=\bigcup\limits_{k=1}^{\infty} V_k$ as increasing unions of open subsets in $X$, satisfying 
	\[U_k\subseteq \overline{U_k}\subseteq V_k\subseteq\overline{V_k}\subseteq D_{-1}\subsetneq X,\]
	and
	\[\theta(U_k)\subseteq \overline{\theta(U_k)}\subseteq \theta(V_k)\subseteq\overline{\theta(V_k)}\subseteq D_1.\]
	Fix $k\in\mathbb{N}$, and let $\alpha^{(k)}=(\{E_n\}_{n\in\mathbb{Z}}, \{\alpha_n\}_{n\in\mathbb{Z}})$ and $\beta^{(k)} =(\{K_n\}_{n\in\mathbb{Z}}, \{\beta_n\}_{n\in\mathbb{Z}})$ be the partial actions of $\mathbb{Z}$ on $X$ generated by the partial homeomorphisms $\alpha_{1}\coloneqq\theta|_{V_k}:V_k\to\theta(V_k)$ and $\beta_1\coloneqq \theta|_{U_k}:U_k\to \theta(U_k)$, respectively.
	Since
	\[C_0(X)\rtimes_\theta\mathbb{Z}= \varinjlim_k C_0(X)\rtimes_{\beta^{(k)}}\mathbb{Z},  \]
	relying on \autoref{p:PerPropDimNucDr}(1) and \autoref{c:IntCases}(2), it is enough to show that $\bigcap_{n\in\mathbb{Z}}\overline{K_n}=\emptyset$. We first observe that $\bigcap_{n\in\mathbb{Z}} E_n$ is an invariant set for our (original) partial automorphism $\theta$. Indeed,
	\[\theta(D_{-1}\cap \bigcap_{n\in\mathbb{Z}} E_n) = \theta(\bigcap_{n\in\mathbb{Z}} E_n)=\alpha_1(\bigcap_{n\in\mathbb{Z}} E_n)= \bigcap_{n\in\mathbb{Z}} E_n, \]
	where at second step we use that $E_{-1}=V_k$. Similarly, one checks invariance under $\theta^{-1}$. By minimality of $\theta$, the assumption that $D_{-1}\neq X$, and \autoref{o:ClosInvSets}, we get $\bigcap_{n\in\mathbb{Z}}E_n=\emptyset$. By construction, \[\overline{K_1}=\overline{\theta(U_k)}\subseteq\theta(V_k)= E_1\] 
	and 
	\[\overline{K_{-1}}=\overline{U_k}\subseteq V_k=E_{-1}.\]
	Thus, \autoref{l:ClosOfDomCont} implies that $\bigcap_{n\in\mathbb{Z}} \overline{K_n}\subseteq \bigcap_{n\in\mathbb{Z}} E_n=\emptyset$, as desired.
\end{proof}

\begin{Cor}\label{c:classificationMinimal}
	Let  $\theta=(\{D_n\}_{n\in\mathbb{Z}}, \{\theta_n \}_{n\in\mathbb{Z}})$ be a minimal partial automorphism acting on an infinite compact Hausdorff second countable space $X$ with finite covering dimension. Then $C(X)\rtimes_{\theta} \mathbb{Z}$ is classifiable.
\end{Cor}

\begin{proof}
	The result for global actions follows from \cite{TW}. If the system is not global, we may assume without loss of generality that $D_{-1}\subsetneq X$. Notice that minimality implies freeness of the partial automorphism. Indeed, if  $x\in D_{-n}$ and $\theta_{-n}(x)=x$ for some $n\in\mathbb{N}$, then 
	\[\{x,\theta(x),\ldots,\theta_{n-1}(x)\}\]
	is a closed $\theta$-invariant set. As $X$ is infinite, this contradicts minimality of $\theta$.
	It follows from \autoref{t:MinimalAut}, \cite[Corollary~29.8]{Exel}, and \cite[Corollary~C]{CETWW} that $C(X)\rtimes_{\theta}\mathbb{Z}$ is a unital simple separable $C^*$-algebra with nuclear dimension 0 or 1.
	To see that $C(X)\rtimes_\theta \mathbb{Z}$ satisfies the UCT, observe that $(C(X)\rtimes_\theta\mathbb{Z}), C(X))$ is a Cartan pair (see \cite{R} for definition) and apply the main result of \cite{BL}.
	We conclude by \cite{EGLN} that $C(X)\rtimes_\theta\mathbb{Z}$ is classifiable by the Elliott invariant.
\end{proof}

\section{Zero dimensional boundaries}\label{s:ZeroDimBdrsSubR}

In this section, we bound the nuclear dimension of crossed product associated to partial automorphisms on a class of 1-dimensional spaces. This class contains all 1-dimensional CW complexes.\\

We start by recalling some topological definitions.

\begin{Def}\label{d:RegOp}
	Let $X$ be a topological space and let $U\subseteq X$ be an open subset. We say that $U$ is \textit{regular open} in $X$ if $\mathrm{int}(\overline{U})=U$.
\end{Def}

\begin{Rmk}\label{r:PrpRegOp}
	\mbox{}
	\begin{enumerate}
		\item  $\mathrm{int}(\overline{A})$ is regular open for any $A\subseteq X$.
		\item Any finite intersection of regular open sets is regular open.
	\end{enumerate}
\end{Rmk}

\begin{Def}\label{d:SemReg}
	A topological space $X$ is called \textit{semiregular} if the regular open subsets in $X$ form a base for the topology.
\end{Def}

We remark that every regular space is semiregular.

\begin{Def}\label{d:HerLind}
	A topological space $X$ is called \emph{Lindel\"{o}f} if every open cover of $X$ has a countable subcover. Moreover, $X$ is called \emph{hereditarily Lindel\"{o}f} if every subspace of it is Lindel\"{o}f.
\end{Def}

\begin{lma}\label{l:RegOpUni}
	Let $X$ be a regular hereditarily Lindel\"{o}f topological space, and let $U$ be an open subset of $X$. Then there exists a sequence $(U_i)_{i=1}^{\infty}$ of regular open subsets in $X$, such that $\overline{U_i}\subseteq U$ for all $i$, and $U=\bigcup_{i=1}^{\infty}U_i$.
\end{lma}

\begin{proof}
	Let $\mathcal{B}=\{W_j \}_{j\in J}$ be a base for the topology, consisting of regular open sets. Given $x\in U$, use regularity of $X$ to find an open neighbourhood $V_x$ of $x$, such that $\overline{V_x}\subseteq U$. Let $j_x\in J$ be such that $x\in W_{j_x}\subseteq V_x$. Then $\{W_{j_x}\}_{x\in U}$ is an open cover of $U$, which admits a countable subcover $(U_i)_{i=1}^{\infty}$, since $X$ is hereditary Lindel\"{o}f. This cover satisfies the required properties.
	
\end{proof}

\begin{lma}\label{l:IncUniRegOp}
	Let $X$ be a regular hereditary Lindel\"{o}f topological space, and let $\theta: U\to V$ be a homeomorphism between two open subsets of $X$. Then there exist increasing sequences $(U_i)_{i=1}^{\infty}$ and $(V_i)_{i=1}^{\infty}$ of regular open sets in $X$ such that
	\begin{enumerate}
		\item $V_i=\theta(U_i)$ for all $i\in \mathbb{N}$,
		\item $U=\bigcup_{i=1}^{\infty}U_i$,
		\item $\overline{U_i}\subseteq U$ for all $i\in\mathbb{N}$, and
		\item $\overline{V_i}\subseteq V$ for all $i\in\mathbb{N}$.
	\end{enumerate}
\end{lma}

\begin{proof}
	Apply \autoref{l:RegOpUni} to obtain a sequence $(U_i')_{i=1}^{\infty}$ of regular open subsets in $X$ such that $U=\bigcup_{i=1}^{\infty}U_i'$, and $\overline{U_i'}\subseteq U$ for all $i\in\mathbb{N}$. For each $i\in\mathbb{N}$, apply \autoref{l:RegOpUni} again to obtain a sequence $(V_j^{(i)})_{j=1}^{\infty}$ of regular open subsets in $X$, such that $\overline{V_j^{(i)}}\subseteq \theta(U_i')$, and $\theta(U_i')=\bigcup_{j=1}^{\infty}V_j^{(i)}$. Now, $$\mathcal{W}\coloneqq \{\theta^{-1}(V_j^{(i)})\colon i,j\in\mathbb{N} \}$$
	is a countable open cover of $U$. We claim that for each $W\in\mathcal{W}$
	\begin{enumerate}
		\item[(a)] $W$ is regular open in $X$,
		\item[(b)] $\overline{W}\subseteq U$, 
		\item[(c)] $\theta(W)$ is regular open in $X$,
		\item[(d)] $\overline{\theta(W)}\subseteq V$.
	\end{enumerate}
	
	Indeed, let $W=\theta^{-1}(V_j^{(i)})\in \mathcal{W}$, for some $i,j\in\mathbb{N}$.
	\begin{enumerate}
		\item[(a)] We have
		$$\mathrm{int}(\overline{\theta^{-1}(V_j^{(i)})})=\mathrm{int}(\theta^{-1}(\overline{V_j^{(i)}}))=\theta^{-1}(\mathrm{int}(\overline{V_j^{(i)}}))=
		\theta^{-1}(V_j^{(i)}), $$
		where at the first step we use that $\overline{V_j^{(i)}}\subseteq V$ and that 
		$$\overline{\theta^{-1}(V_j^{(i)})}\subseteq \overline{\theta^{-1}(\theta(U_i'))}= \overline{U_i'}\subseteq U; $$
		at the second step we use that $U,V$ are open in $X$; at the last step we use that $V_j^{(i)}$ is regular open.
		\item[(b)] The fact that $\overline{W}\subseteq U$ is part of the proof of (a).
		\item[(c)]  Indeed, $\theta(W)=V_j^{(i)}$ is regular open.
		\item[(d)] Follows since $\overline{\theta(W)}=\overline{V_j^{(i)}}\subseteq V$, for all $i,j\in\mathbb{N}$.
	\end{enumerate}
	Enumerate $\mathcal{W}=(W_i)_{i=1}^{\infty}$. We define inductively an increasing sequence $(U_i)_{i=1}^{\infty} $ of regular open subsets in $X$ that satisfy conditions (1)-(4) in the lemma. Set $U_1\coloneqq W_1$. Fix $n> 1$, and assume that we have defined $(U_i)_{i=1}^{n-1}$. Set
	$$U_n\coloneqq \mathrm{int}(\overline{U_{n-1}\cup W_n}).$$
	By \autoref{r:PrpRegOp}(1), $U_n$ is regular open in $X$. Moreover, by condition (b) for $W_n$, and the inductive assumption that $\overline{U_{n-1}}\subseteq U$, we have 
	$$\overline{U_n}\subseteq \overline{U_{n-1}\cup W_n}\subseteq U.$$
	Next, 
	$$\theta(U_n)=\theta(\mathrm{int}(\overline{U_{n-1}\cup W_n}))=\mathrm{int}(\theta(\overline{U_{n-1}\cup W_n})) =\mathrm{int}(\overline{\theta(U_{n-1}\cup W_n)}),$$
	where at the second step we use that $U,V$ are open in $X$, and at the last step we use that $\overline{U_{n-1}\cup W_n}\subseteq U$, and that 
	$\overline{\theta(U_{n-1}\cup W_n)}\subseteq  V$
	by the inductive assumption that $\overline{\theta(U_{n-1})}\subseteq V$ and condition (d) for $W_n$. \autoref{r:PrpRegOp} implies that $\theta(U_n)$ is regular open in $X$. Lastly, observe that it follows from the above computation that
	$$\overline{\theta(U_n)}\subseteq \overline{\theta(U_{n-1}\cup W_n)}\subseteq V.$$
	We complete the proof by setting $V_i\coloneqq \theta(U_i)$ for all $i\in\mathbb{N}$.
	
\end{proof}

\begin{Rmk}\label{r:RegOpDoms}
	Let $X$ be a topological space, and let $h: U\to V$ be a homeomorphism between two open subsets of $X$. Assume that $h$ restricts to a homeomorphism $\theta: D_{-1}\to D_1$, where $D_{-1}$ and $D_1$ are regular open subsets of $X$, $\overline{D_{-1}}\subseteq U$, and $\overline{D_1}\subseteq V$. Then $\theta$ generates a toplogical partial action $\Theta=(\{D_n\}_{n\in\mathbb{Z}}, \{\theta_n \}_{n\in\mathbb{Z}})$ with regular open domains.
\end{Rmk}

\begin{proof}
	We prove by induction that $D_n$ and $D_{-n}$ are regular open in $X$, for all $n\geq 1$. For $n=0,1$ this follows from the assumptions. Let $n\geq 2$, and assume that $D_{n-1}$ and $D_{-(n-1)}$ are regular open in $X$. Then
	
	\begin{align*}
	D_n&=\theta_1(D_{n-1}\cap D_{-1})\\
	&=\theta(\mathrm{int}(\overline{D_{n-1}\cap D_{-1}}))\\
	&=\mathrm{int}(h(\overline{D_{n-1}\cap D_{-1}}))\\
	&=\mathrm{int}(\overline{h(D_{n-1}\cap D_{-1})})\\
	&=\mathrm{int}(\overline{D_n}),
	\end{align*}
	at the second step we use the inductive assumption and \autoref{r:PrpRegOp}(2), at the third step we use $\overline{D_{n-1}\cap D_{-1}}\subseteq U$, and at the fourth step we use $\overline{h(D_{n-1}\cap D_{-1})}\subseteq V$. The argument for $D_{-n}$ is similar.
\end{proof}

\begin{Prop}\label{p:ResParAutRegOpDoms}
	Let $X$ be a \LCH second countable space, and let $\theta: D_{-1}\to D_1$ be a partial automorphism on $X$. Assume that the partial action generated by $\theta$ has regular open domains $\{D_n\}_{n\in\mathbb{Z}}$. Then
	\[\dimnuc(C_0(X)\rtimes_{\theta}\mathbb{Z})\leq \dimnuc(C_0(\partial_X\big(\bigcap_{n\in\mathbb{Z}}\overline{D_n}\big))\rtimes \mathbb{Z}) +2\dim(X)^2 +8\dim(X)+7. \]
\end{Prop}

\begin{proof}
	By \autoref{t:ResIntDom}, we know that it suffices to show that 
	\[\dimnuc(C_0\big(\bigcap_{n\in\mathbb{Z}}\overline{D_n}\big))\rtimes\mathbb{Z})\leq \dimnuc(C_0(\partial_X\big(\bigcap_{n\in\mathbb{Z}}\overline{D_n}\big))\rtimes \mathbb{Z}) + 2\dim(X)^2 +6\dim(X)+5. \]
	Notice that $\mathrm{int}_X\big(\bigcap_{n\in\mathbb{Z}}\overline{D_n}\big)$ is an open, $\theta$-invariant set inside $\bigcap_{n\in\mathbb{Z}}\overline{D_n}$. Thus
	\[0\to C_0(\mathrm{int}_X\big(\bigcap_{n\in\mathbb{Z}}\overline{D_n}\big))\rtimes\mathbb{Z}\to C_0\big(\bigcap_{n\in\mathbb{Z}}\overline{D_n}\big)\rtimes\mathbb{Z}\to C_0(\partial_X \big(\bigcap_{n\in\mathbb{Z}}\overline{D_n}\big))\rtimes\mathbb{Z}\to 0 \]
	is a short-exact sequence of $C^*$-algebras. 
	For every $k\in\mathbb{Z}$ we have
	\[\mathrm{int}_X\big(\bigcap_{n\in\mathbb{Z}}\overline{D_n}\big)\subseteq \mathrm{int}(\overline{D_k})=D_k, \]
	because $D_k$ is regular open in $X$. Therefore, $\theta$ restricts to a global action on $\mathrm{int}_X(\bigcap_{n\in\mathbb{Z}}\overline{D_n})$. By \autoref{t:HirWu}
	\[\dimnuc(C_0(\mathrm{int}_X\big(\bigcap_{n\in\mathbb{Z}}\overline{D_n}\big))\rtimes\mathbb{Z})\leq 2\dim(X)^2+6\dim(X)+4.  \]
	Applying \autoref{p:PerPropDimNucDr} to the short exact sequence, we get the required bound.
\end{proof}

The next Proposition shows that every partial automorphism on a \LCH second countable space is a direct limit of partial automorphisms with regular open domains.

\begin{Prop}\label{p:ParAutLimRegOpDom}
	Let $X$ be a \LCH second countable space, and let $\theta: U\to V$ be a partial automorphism on $X$. 
	Then, there exists a sequence $\{\theta^{(i)}\colon U_i\to V_i\}$ of partial automorphisms on $X$ with regular open domains such that
	$$C_0(X)\rtimes_{\theta}\mathbb{Z}=\varinjlim_i C_0(X)\rtimes_{\theta^{(i)}}\mathbb{Z}.$$
\end{Prop}

\begin{proof}
	$X$ is a metrizable second countable space and in particular also regular hereditary Lindel\"{o}f space. Thus, we can apply \autoref{l:IncUniRegOp} for $U,V,\theta$ and obtain increasing sequences $(U_i)_{i=1}^{\infty}$ and $(V_i)_{i=1}^{\infty}$ that satisfy conditions (1)-(4) in the Lemma. For each $i$, let $\theta^{(i)}: U_i\to V_i$ be the restriction of $\theta$ to $U_i$. By \autoref{lma:DirectLimitPar} we have
	$$C_0(X)\rtimes_{\theta}\mathbb{Z}=\varinjlim_i C_0(X)\rtimes_{\theta^{(i)}}\mathbb{Z}.$$
	Moreover, \autoref{r:RegOpDoms} implies that $\theta^{(i)}$ generates a partial action with regular open domains.
\end{proof}

\begin{Def}\label{d:ClassC}
	Let $\mathcal{C}$ denote the class of all $1$-dimensional \LCH second countable topological spaces with the following property: If $X\in \mathcal{C}$ and $F\subseteq X$ is a closed subspace in $X$, then $\dim(\partial_X F)=0$.
\end{Def}

\begin{Obs}\label{o:ClosedInC}
	If $X\in\mathcal{C}$ and $Y\subseteq X$ is a $1$-dimensional closed subspace, then $Y\in\mathcal{C}$.
\end{Obs}

\begin{Exl}\label{e:ExlCMan}
	Every $1$-manifold belongs to the class $\mathcal{C}$ (see \cite[Corollary~20]{Sch}).
\end{Exl}

Also every locally compact Hausdorff second countable graph belongs to the class $\mathcal{C}$. We recall the definition.
A \textit{graph} $X$ is a topological space which arises from a usual graph $G=(E,V)$ by replacing isolated vertices $X_0\subseteq V$ by points and each edge $e\in E$ by a copy of the unit interval, where $0$ and $1$ are identified with the endpoints of $e$. Equip $X$ with the quotient topology of the set $X_{0}\sqcup \bigsqcup_{e\in E}[0,1]_{e}$ under the quotient map $q$ used for gluing.

\begin{Obs}\label{o:VerDisc}
	Let $X$ be a graph. Then the induced topology on $V$ is discrete.
\end{Obs}

\begin{proof}
	We show that for every $x\in V$, the singleton $\{x\}$ is open in $V$. If $x$ is an isolated vertex, then $q^{-1}(\{x\})\in X_0$ and $X_0$ is a discrete space which $q$ leaves as it is, so $\{x\}$ is open even as a subset of $X$. Assume now $x\notin X_0$. 
	Set
	\[U\coloneqq \bigcup\limits_{z\in V, (xz)\in E}[0,0.5)_{(xz)} \cup \bigcup\limits_{z\in V, (zx)\in E} (0.5,1]_{(zx)}.  \]
	Notice that $U$ is open in $X_0\cup \bigsqcup_{e\in E}[0,1]_e$ and $q^{-1}(q(U))=U$. Therefore, $q(U)$ is open in $X$. As $q(U)\cap V=\{x\}$, the claim follows.
\end{proof}

\begin{lma}\label{l:CGraph}
	Let $X$ be a locally compact Hausdorff second countable graph and let $M$ be a subset of $X$ such that $\mathrm{int}(M)=\emptyset$. Then $\dim(M)=0$.
\end{lma}

\begin{proof}
	We first claim that $\mathrm{int}(q^{-1}(M))=\emptyset$. Observe that $\mathrm{int}(q^{-1}(M))\subseteq \bigsqcup_{e\in E}[0,1]_{e}$, since by assumption $M$ contains no isolated vertices. Thus, if $x\in \mathrm{int}(q^{-1}(M))\cap I_e$, there exists an open neighborhood $U$ of $x$ such that $U\subseteq \mathrm{int}(q^{-1}(M))\cap I_e$. By shrinking $U$ if necessary, $\mathrm{int}(q^{-1}(M))$ contains an open subinterval of $(0,1)_e$. This contradicts the assumption that $\mathrm{int}(M)=\emptyset$, and the claim follows.
	For any $e\in E$, the quotient map restricts to a homeomorphism $q|_{(0,1)_e}:(0,1)_e\to q((0,1)_e)$. Thus, $M\cap q\big(\bigsqcup_{e\in E}(0,1)_e\big)\cong q^{-1}(M)\cap \bigsqcup_{e\in E}(0,1)_e$ via $q$. This implies
	\[\dim\big(M\cap q\big(\bigsqcup_{e\in E}(0,1)_e\big)\big)=\dim\big(q^{-1}(M)\cap \bigsqcup_{e\in E}(0,1)_e\big)=0, \]
	where the equality to $0$ follows from \cite[Corollary~20]{Sch} and the fact that $q^{-1}(M)\cap \bigsqcup_{e\in E}(0,1)_e$ has an empty interior (as a subspace of $\bigsqcup_{e\in E}(0,1)_e$). We have
	\[M=\big(M\cap q\big(\bigsqcup_{e\in E}(0,1)_e\big)\big)\cup \underbrace{\big(M\cap\bigcup_{e\in E} q(\{0,1\}_e)\big)}_{M\cap V}.\]
	By \autoref{o:VerDisc}, $M\cap V$ is a discrete space. Thus, we viewed $M$ as a union of a closed and an open zero-dimensional subsets. In a metrizable space, every open set is an $F_\sigma$-set. Applying \autoref{t:DimF-Sig}, we obtain $\dim(M)=0$.
\end{proof}

The following is immediate from the Lemma.

\begin{Exl}
	Every locally compact Hausdorff second countable graph belongs to $\mathcal{C}$.
\end{Exl}

Next, we give an example of a $1$-dimensional \LCH second-countable space that does not belong to the class $\mathcal{C}$.

\begin{Exl}\label{e:NonExlC}
	$X=\left(\{0\}\cup \{\frac{1}{n}: n>0 \}\right)\times [0,1]$ with the product topology is a one-dimensional space and the closed subspace $F=\{0\}\times [0,1]$ has a one-dimensional boundary $\partial_X F=F$.
\end{Exl}

As an application to the techniques developed in this section, we bound the nuclear dimension of crossed products associated to partial automorphisms on locally compact subspaces of topological spaces belonging to $\mathcal{C}$ (e.g. locally compact subspaces of $\mathbb{R}$ with the induced Euclidean topology).

\begin{Cor}\label{c:ResPartAutRRegOpDoms}
	Let $Y\in\mathcal{C}$ and $X\subseteq Y$ be a locally compact subspace. Let $\theta: D_{-1}\to D_1$ be a partial automorphism on $X$. Then
	\[\dimnuc(C_0(X)\rtimes_{\theta}\mathbb{Z})\leq  2\dim(X)^2 +8\dim(X)+11. \]
\end{Cor}

\begin{proof}
	\autoref{t:DimSub} implies $\dim(X)\leq 1$. In the case where $\dim(X)=0$, we are done by \autoref{t:ResZeroDim}, so assume $\dim(X)=1$.
	First we treat the case that $X$ is closed inside $Y$. 
	Since nuclear dimension respects direct limits, \autoref{p:ParAutLimRegOpDom} allows to assume that the partial action generated by $\theta$ has regular open domains $\{D_n\}_{n\in\mathbb{Z}}$.
	As $X\in\mathcal{C}$ by \autoref{o:ClosedInC}, we have
	$$\dim\big(\partial_X\big(\bigcap_{n\in\mathbb{Z}}\overline{D_n}\big)\big)=0.$$
	Combine \autoref{t:ResZeroDim} with \autoref{p:ResParAutRegOpDoms} for the desired bound. If $X$ is not closed, denote by $\mathrm{cl}_Y (X)$ the closure of $X$ inside $Y$. Since $X$ is a locally compact subset of $Y$, it is also locally closed in $Y$ (that is, $X$ is open inside $\mathrm{cl}_Y (X)$). Therefore, we can regard $D_{-1}$ and $D_1$ as open subsets inside $\mathrm{cl}_Y(X)$, and $\theta:D_{-1}\to D_1$ as a partial automorphism generating a partial action on $\mathrm{cl}_Y(X)$. We have
	$$0\to C_0(X)\rtimes\mathbb{Z}\to C_0(\mathrm{cl}_Y(X))\rtimes\mathbb{Z}\to C_0(\mathrm{cl}_Y(X)\setminus X)\rtimes\mathbb{Z}\to 0,$$
	a short exact sequence of $C^*$-algebras. Thus,
	$$\dimnuc(C_0(X)\rtimes_\theta\mathbb{Z})\leq \dimnuc(C_0(\mathrm{cl}_Y(X))\rtimes_{\theta}\mathbb{Z}),$$
	and we can use the bound obtained for closed subsets inside $Y$, since $$\dim(\mathrm{cl}_Y(X))=\dim(X)=1.$$
\end{proof}

\end{document}